\documentclass[11pt]{amsart}

\usepackage{amsmath,amssymb,amsthm}
\usepackage{a4wide}
\usepackage{hyperref}

\usepackage{graphicx}
\usepackage{xypic}
\entrymodifiers={+!!<0pt,\fontdimen22\textfont2>}
\usepackage[all]{xy}

\newtheoremstyle{myremark} 
    {7pt}                    
    {7pt}                    
    {}  	                 
    {}                           
    {\bf}       	         
    {.}                          
    {.5em}                       
    {}  

\theoremstyle{plain}
\newtheorem{lemma}{Lemma}[section]
\newtheorem{theorem}[lemma]{Theorem}

\newtheorem{definition}[lemma]{Definition}

\newtheorem{proposition}[lemma]{Proposition}
\newtheorem{conjecture}[lemma]{Conjecture}
\newtheorem{claim}{Claim}

\theoremstyle{myremark}
\newtheorem{remark}[lemma]{Remark}

\newcommand{\zet}{\mathbb{Z}}

\renewcommand{\subset}{\subseteq}

\newcommand{\lk}{\mathrm{lk}}

\newcommand{\Michal}[1]{}
\newcommand{\Honza}[1]{}
\newcommand{\HONZA}[1]{}

\begin{document}
\title[Dense flag triangulations of $3$-manifolds]{Dense flag triangulations of $3$-manifolds via extremal graph theory}
\author[Micha{\l} Adamaszek]{Micha{\l} Adamaszek}
\address{Fachbereich Mathematik, Universit\"at Bremen
      \newline Bibliothekstr. 1, 28359 Bremen, Germany}
\email{aszek@mimuw.edu.pl}
\author[Jan Hladk\'y]{Jan Hladk\'y}
\address{Mathematics Institute and DIMAP,
      \newline University of Warwick, Coventry, CV4 7AL, UK}
\email{honzahladky@gmail.com}
\thanks{Research of MA was carried out while a member of the Centre for Discrete
        Mathematics and its Applications (DIMAP), supported by the EPSRC award EP/D063191/1. JH
        is an EPSRC Research Fellow.}


\begin{abstract}
We characterize $f$-vectors of sufficiently large three-dimensional flag
Gorenstein$^*$ complexes, essentially confirming a conjecture of Gal~[Discrete Comput. Geom., 34 (2), 269--284, 2005]. In
particular, this characterizes $f$-vectors of large flag triangulations of the $3$-sphere. Actually, our main result is more general and describes the structure of closed flag 3-manifolds which have many edges. 

Looking at the 1-skeleta of these manifolds we reduce the problem to a certain
question in extremal graph theory. We then resolve this question by employing the Supersaturation Theorem of Erd\H{o}s and Simonovits.
\end{abstract}
\maketitle

\section{Introduction}
\label{section:intro}

One of the trends in enumerative combinatorics is to classify face numbers of
various families of simplicial complexes. In this paper we study flag triangulations of closed $3$-manifolds with sufficiently many vertices and high edge density. As a consequence we confirm, for sufficiently large number of vertices, a conjecture of Gal regarding face vectors of flag triangulations of generalized homology $3$-spheres. 

If $K$ is a finite simplicial complex and $\sigma \in K$ is a face we denote by $|\sigma|$ its number of vertices and by $\dim\sigma=|\sigma|-1$ its dimension. The dimension of $K$, $\dim K$, is the maximum over all $\sigma\in K$ of $\dim \sigma$.

The \emph{$f$-vector} of a simplicial complex $K$ of dimension $d$ is the
sequence
\begin{equation}
(f_{-1},f_0,\ldots,f_d)
\end{equation}
where $f_i$ is the number of faces of dimension $i$. By convention, we always define $f_{-1}=1$. The
\emph{$h$-vector} of $K$ is the sequence
\begin{equation}
(h_0,\ldots,h_{d+1})
\end{equation}
determined by the equation\footnote{Note that we consistently use $d$ for the dimension of $K$, rather than the cardinality of its largest face, hence the indices and exponents in most formulae are shifted by $1$ compared to what they usually look like.}
\begin{equation}
\sum_{i=0}^{d+1} h_ix^{d+1-i}=\sum_{i=-1}^d f_{i}(x-1)^{d-i}.
\end{equation}
Of course the $f$-vector and the $h$-vector determine one another and carry the
same information, but the $h$-vector often enjoys better combinatorial properties; the Dehn-Sommerville equation~\eqref{eq:DS} below being one example. Note that $h_0=1$.

Next we introduce the class of Gorenstein$^*$ and Eulerian complexes. The reader not interested in this level of generality can equally well think about simplicial complexes which triangulate a standard sphere. Recall that if $\sigma\in K$ is a face then the \emph{link} of $\sigma$ in $K$, denoted $\lk_K\sigma$ is the subcomplex $\{\tau\in K~|~\tau\cap\sigma=\emptyset,\ \tau\cup\sigma\in K\}$.

A simplicial complex $K$ of dimension $d$ is a \emph{generalized homology sphere} (or \emph{Gorenstein$^*$ complex}) if for every face $\sigma\in K$ the homology of $\lk_K \sigma$ is the same as the homology of a sphere of dimension $d-|\sigma|$. In particular, when $\sigma=\emptyset$, this means that $K$ itself has the homology of a $d$-sphere. We are going to use the short name `$d$-GHS'. A simplicial complex $K$ of dimension $d$ is \emph{Eulerian} if for every face $\sigma\in K$ the Euler characteristic of $\lk_K \sigma$ is the same as that of a sphere of dimension $d-|\sigma|$.

Any triangulation of the standard $d$-sphere is a $d$-GHS and every $d$-GHS is Eulerian. More generally, if $K$ is a triangulation of a closed topological manifold and $\sigma\neq\emptyset$ is a face of $K$ then $\lk_K\sigma$ is a $(d-|\sigma|)$-GHS. By the Poincar\'e duality the Euler characteristic of an odd-dimensional closed manifold is $0$, hence every such manifold is Eulerian. (A closed manifold means a compact manifold without boundary.)

Any Eulerian complex of dimension $d$ satisfies the classical \emph{Dehn-Sommerville equations}
\begin{equation}\label{eq:DS}h_i=h_{d+1-i}\end{equation}
and, following Gal~\cite{SRGal}, one can encode the coefficients $h_i$ in a
shorter, integer-valued \emph{$\gamma$-vector}
\begin{equation}
(\gamma_0,\ldots,\gamma_{\lfloor\frac{d+1}{2}\rfloor})
\end{equation}
determined by the equation
\begin{equation}
\sum_{i=0}^{d+1} h_ix^i=\sum_{i=0}^{\lfloor\frac{d+1}{2}\rfloor} \gamma_{i}x^i(x+1)^{d+1-2i}.
\end{equation}
We always have $\gamma_0=1$.

The classification of $h$- (or $f$-, $\gamma$-) vectors of generalized homology spheres is of great interest in the field. The complete classification is predicted by the celebrated $g$-conjecture of McMullen~\cite{McMullen71}. In this work we pick up a related but somewhat different research line started by Gal, who investigated these parameters for the restricted family of flag complexes.

A simplicial complex is \emph{flag} if all its minimal non-faces have dimension $1$ or, equivalently, if it is the clique complex of its $1$-skeleton. The latter means that faces of $K$ correspond to cliques in $K^{(1)}$, the graph which is the $1$-dimensional skeleton of $K$. For flag generalized homology spheres the $\gamma$-vector is the most efficient and interesting parameter. The major conjecture of Gal~\cite[Conj. 2.1.7]{SRGal}, which states that the $\gamma$-vector of a flag $d$-GHS is non-negative, is known to hold for $d\leq 4$ \cite[Cor.2.2.3]{SRGal}. For any flag $(2d-1)$-GHS this conjecture is a strengthening of the famous Charney-Davis conjecture \cite{CharneyDavis}. On the other hand, Gal's conjecture itself has a stronger version which states that the $\gamma$-vector of a flag $d$-GHS is an $f$-vector of some flag complex \cite{NevoPetersen}. See \cite{NevoPetTenner} and references therein for progress in that area.

If $K$ and $L$ are two simplicial complexes with disjoint vertex sets then their \emph{join} $K\ast L$ is a simplicial complex with vertex set $V(K)\cup V(L)$ whose faces are all unions $\tau\cup\sigma$ for $\tau\in K$, $\sigma\in L$. It is a standard fact that $S^k\ast S^l=S^{k+l+1}$ for triangulated spheres $S^k$, $S^l$ with $k,l\geq -1$.

Following Murai and Nevo~\cite{MuraiNevo}, let $\Lambda_d$ denote the set of all $\gamma$-vectors of flag $d$-GHSs. When $d=1,2$ then the $(k+4)$-gon or its join with the two-point sphere $S^0$ are simplicial $d$-spheres with $\gamma$-vector $(1,k)$ for any integer $k\geq 0$, and by the previous discussion these exhaust $\Lambda_1$ and $\Lambda_2$, i.e., we have $$\Lambda_1=\Lambda_2=\left\{(1,k)\in\zet^2\::\:  k\ge 0\right\}\;.$$ Gal~\cite[Cor. 3.1.7]{SRGal} proved that $\gamma_2\leq\gamma_1^2/4$ must hold for any $\gamma$-vector $(1,\gamma_1,\gamma_2)$ in $\Lambda_3$ or $\Lambda_4$ and a simple join construction \cite[Thm. 5.1.ii]{MuraiNevo} shows that this is tight in dimension $4$, that is
\begin{equation*}
\Lambda_4\ =\ \left\{(1,\gamma_1,\gamma_2)\in\zet^3\::\: \gamma_2\leq
\frac{\gamma_1^2}{4}, \quad \gamma_1,\gamma_2\geq 0\right\}.
\end{equation*}
Going back to dimension $3$, Gal~\cite[Thm. 3.2.1]{SRGal} showed that
\begin{equation}\label{eq:inclusion}
\Lambda_3 \ \supseteq \ \left\{(1,\gamma_1,\gamma_2)\in\zet^3\::\: \gamma_2\leq
\frac{(\gamma_1-1)^2}{4}, \quad \gamma_1,\gamma_2\geq 0\right\}\ \cup
\left\{(1,k+l,kl)\in\zet^3\::\: k,l\geq 0\right\}.
\end{equation}
The elements of the first set can be realized as $\gamma$-vectors of some
appropriate iterated edge subdivisions of the boundary of the cross-polytope.
The elements of the second kind are the $\gamma$-vectors of a join of a
$(k+4)$-gon with an $(l+4)$-gon. 

Gal then conjectures that
the inclusion~\eqref{eq:inclusion} is in fact an equality. Since the $\gamma$-vector of a flag $3$-GHS is non-negative, the stronger version of that conjecture is the following (see \cite[Con. 3.2.2]{SRGal} or \cite[Conj. 5.2]{MuraiNevo}).
\begin{conjecture}\label{conj:galmain}
If $(1,\gamma_1,\gamma_2)$ is the $\gamma$-vector of a flag $3$-GHS $K$ and $\gamma_2>\frac{(\gamma_1-1)^2}{4}$ then $K$ is a join of two polygons.
\end{conjecture}
Also, note that the two constructions which show the
inclusion~\eqref{eq:inclusion} are flag triangulations of the $3$-sphere. Thus --- if true ---
Conjecture~\ref{conj:galmain} provides a characterization of
$\gamma$-vectors (or $f$-vectors) of flag triangulations of the $3$-sphere. 
Even this special case of characterization of
$\gamma$-vectors of flag triangulations of the $3$-sphere is open.
The conjecture was verified for order complexes of posets~\cite{MuraiNevo}. 

To make the following discussion more concrete, suppose that $K$ is an Eulerian complex of dimension $3$ with face numbers $(1,f_0,f_1,f_2,f_3)$. Then the Dehn-Sommerville relations translate into
\begin{equation}
\label{eq:f2f3}
f_2=2(f_1-f_0),\quad f_3=f_1-f_0.
\end{equation}
Moreover, we find
\begin{equation}
\label{eq:gammas}
\gamma_1=f_0-8,\quad \gamma_2=f_1-5f_0+16
\end{equation}
and the conditions $(\gamma_1-1)^2/4<\gamma_2\leq \gamma_1^2/4$ are equivalent to
\begin{equation}
\label{eq:turanbound}
\frac{1}{4}(f_0^2+2f_0+17)<f_1\leq \frac{1}{4}f_0^2+f_0.
\end{equation}

\subsection*{Our results}
Below is the main result of the paper. It determines the structure of closed flag $3$-manifolds which have many edges.
\begin{theorem}
\label{thm:main}
There exists a number $n_0$ such that the following holds. If $M$ is a flag triangulation of a closed $3$-manifold with $f_0\geq n_0$ vertices, $f_1$ edges, and such that $f_1>\frac14(f_0^2+2f_0+17)$ then $M$ is a join of two polygons (and, in particular, it is homeomorphic to $S^3$).
\end{theorem}

Theorem~\ref{thm:main} resolves Conjecture~\ref{conj:galmain} affirmatively for flag complexes with sufficiently many vertices because every $3$-GHS is a closed manifold (see Remark~\ref{rmrk:sub-ghs}). In other words, the inclusion \eqref{eq:inclusion} is an equality except for, perhaps, a finite number of elements. 

Below, we prepare tools for our proof of Theorem~\ref{thm:main}. We shall reduce
Theorem~\ref{thm:main} to a certain statement in extremal graph theory
(Theorem~\ref{lem:extremalgraphtheory}).

\medskip

Given a graph $G$ and a vertex $v\in V(G)$ we write $N_v$ for the neighborhood of $v$, that is $\{w\in V(G)~:~vw\in E(G)\}$. If $W\subset V(G)$ then $G[W]$ is the subgraph of $G$ induced by $W$.
The \emph{length of a path in a graph} is its number of vertices; this is one
more than the standard common definition of length but more convenient for our
purposes.

\begin{definition}\label{def:link}
If $G$ is a graph and $\sigma$ is a clique in $G$ then define the \emph{link of
$\sigma$} in $G$ as $$\lk_G\sigma=G\left[\bigcap_{v\in\sigma} N_v\right]\;.$$
That is, $\lk_G\sigma$ is the subgraph of $G$ induced by the vertices which are
not in $\sigma$, but are adjacent to every vertex of $\sigma$.
\end{definition}

Definition~\ref{def:link} is designed so that it is  compatible with the topological notion
of links in flag complexes. For each flag complex $K$ we have
$\lk_{K^{(1)}}\sigma=(\lk_K\sigma)^{(1)}$, where on the left-hand side we use the link 
of Definition~\ref{def:link} and on the right-hand side the link
is understood in the simplicial sense.

Let us define the class of graphs which arise in our setting.
\begin{definition}\label{def:fascinating}
A graph $G$ with $n$ vertices and $m$ edges is \emph{fascinating} if it satisfies the following conditions
\begin{itemize}
\item[a)] $G$ contains exactly $2(m-n)$ triangles.
\item[b)] For every edge $e$ in $G$ the link $\lk_Ge$ is a cycle of length at least $4$.
\item[c)] For every triangle $t$ in $G$ the link $\lk_Gt$ is the discrete graph with $2$ vertices and no edges.
\item[d)] For every vertex $v$ in $G$ the link $\lk_Gv$ is a connected, planar graph whose every face (including the unbounded one) is a triangle. In particular -- by Kuratowski's Theorem -- it does not contain the complete bipartite graph $K_{3,3}$ as a subgraph.

Further, $\lk_Gv$ contains at least 6 vertices.
\end{itemize}
\end{definition}

Our reduction is based on the next observation.
\begin{lemma}
\label{lemma:1skelgood}
If $M$ is a closed flag $3$-manifold then the $1$-skeleton of $M$ is fascinating.
\end{lemma}
\begin{proof}
Let $G=M^{(1)}$. Condition a) follows since $M$ is Eulerian, and so it satisfies (\ref{eq:f2f3}). Parts b)--d) are consequences of the fact that $\lk_Mt$, $\lk_Me$, $\lk_Mv$ are flag triangulations of, respectively, $S^0$, $S^1$ and $S^2$. A known fact that a flag triangulation of $S^j$ requires at least $2(j+1)$ vertices \cite[Lem.2.1.14]{SRGal} proves that the links must be sufficiently large.
%
\end{proof}




The \emph{graph join} of graphs $G$ and $H$, which we will denote $G\ast H$, is the disjoint union of $G$ and $H$ 
together with all the edges between $V(G)$ and $V(H)$. For any simplicial complexes $K$ and $L$ we have 
$(K\ast L)^{(1)}=K^{(1)}\ast L^{(1)}$, where on the left-hand side we use the simplicial join.

By Lemma~\ref{lemma:1skelgood} we get that Theorem~\ref{thm:main} is a consequence of the following result.
\begin{theorem}
\label{lem:extremalgraphtheory} There exists a number $n_0$ such that the
following holds. Suppose $G$ is a fascinating graph with $n\ge n_0$ vertices, $m$ edges
and $m>\frac{1}{4}(n^2+2n+17)$. Then $G$ is a join of two cycles.
\end{theorem}

The rest of the paper is concerned with the proof of this theorem. The strategy is outlined at the beginning of the next section.

\begin{remark}
\label{remark:funnygraphs}
Along the way we will also see that the result is tight in the following sense: There exist flag $3$-spheres with arbitrarily large $f_0$ and with exactly
$$f_1=\frac{1}{4}(f_0^2+2f_0+17)$$
edges, which are not a join of two cycles. Moreover, we will classify those boundary cases: Any fascinating graph $G$ with $n\geq n_0$ vertices and exactly $m=\frac{1}{4}(n^2+2n+17)$ edges is one of the graphs in Figure~\ref{rysunek-equal} in Section~\ref{section:exact}.
\end{remark}

\begin{remark}
\label{remark:upperbound}
Theorem~\ref{thm:main} implies that for $f_0\geq n_0$ every closed flag $3$-manifold satisfies $f_1\leq \frac14 f_0^2+f_0$ (or, equivalently, $\gamma_2\leq\frac14 \gamma_1^2$). This result in fact holds for \emph{all} values of $f_0$ by the same proof that works for $3$-GHSs in \cite{SRGal}.
\end{remark}

\begin{remark}
\label{rmrk:sub-ghs}
In dimensions $d=0,1,2$ the classes of (flag) $d$-spheres and $d$-GHS coincide and in dimension $d=3$ every $3$-GHS is a closed, connected manifold. To see this, first note that it is an easy consequence of the definition that if $L$ is a $d$-GHS and $\sigma\in L$ then $\lk_L\sigma$ is a $(d-|\sigma|)$-GHS.
Now the only $0$-complex with the homology of $S^0$ is $S^0$ itself. As for $d=1$, observe that in a $1$-GHS all vertex links are the two-point space, so a $1$-GHS is a disjoint union of cycles, of which only a single cycle has the homology of $S^1$. In a $2$-GHS the link of every vertex is the sphere $S^1$, so a $2$-GHS is a closed surface, and of all surfaces only $S^2$ has the correct homology. Finally it means that in a $3$-GHS all face links are homeomorphic to spheres of appropriate dimensions, so a $3$-GHS is a closed manifold.
\end{remark}


\section{Proof of Theorem~\ref{lem:extremalgraphtheory}}\label{sec:ProofOfExtremalLemma}

The main idea behind our approach is that $G$ has a lot of edges (more than
$n^2/4$), but relatively few triangles -- just $\Theta(n^2)$. Graphs with this
edge density must have many more triangles, namely $\Theta(n^3)$, unless they
look very ``similar'', in some sense, to the complete bipartite graph
$K_{n/2,n/2}$. This phenomenon is called \emph{supersaturation} and is one of
the basic principles of extremal (hyper)graph theory with fundamental
applications to areas like additive combinatorics or property testing in
computer science. In our setting the additional properties of $G$ coming from
Definition~\ref{def:fascinating} can be used to refine the similarity to
$K_{n/2,n/2}$ to determine the structure of $G$ exactly. This is a
relatively standard approach in Extremal Graph Theory, called the
\emph{Stability method}, and introduced by Simonovits~\cite{S68}. However, our
proof is somewhat more complex than most of the applications of the Stability
method to problems in extremal graph theory. Indeed, in these problems one
usually tries to determine exactly the structure of a unique extremal graph
while here we are dealing with joins of two cycles whose lengths can vary, i.e., graphs with somewhat
looser structure.

Here is a more detailed
outline of the proof.
Mantel's Theorem (which is a special case of Tur\'an's Theorem) asserts that the complete balanced
bipartite graph $K_{\lfloor h/2\rfloor,\lceil h/2\rceil}$ is the unique maximizer of the number edges among all triangle-free
graphs on $h$ vertices. Note that this graph has $\lfloor h^2/4\rfloor$ edges.
The graph $K_{\lfloor h/2\rfloor,\lceil h/2\rceil}$ is \emph{stable} for this
extremal problem in the following sense: if $H$ is a graph on $h$ vertices with at least $h^2/4$
edges and containing only $o(h^3)$ triangles, it must be ``very
similar'' (the precise meaning appears in Theorem~\ref{thm:stability}) to $K_{\lfloor h/2\rfloor,\lceil h/2\rceil}$. 
These conditions are satisfied for the fascinating graph $G$ of Theorem~\ref{lem:extremalgraphtheory}. By exploiting other properties of $G$ we will be able to show that $G$ is close to being a join of two cycles in the sense of the next definition.

\begin{definition}
\label{def:joinlike}
A fascinating graph $G$ is called \emph{$t$-joinlike} if there is a partition $V(G)=C_1\sqcup C_2\sqcup X$ where
\begin{itemize}
\item the graphs $G[C_i]$ are cycles,
\item there are edges $e_i\in G[C_i]$ such that $\lk_G e_i=G[C_{3-i}]$,
\item $|X|=t$.
\end{itemize}
The vertices of $X$ are called \emph{exceptional}.
\end{definition}
Note that a $0$-joinlike fascinating graph is a join of two cycles $G[C_1]\ast G[C_2]$. At the end of this Section we will establish that $G$ must be $t$-joinlike for $t=0$, $1$ or $2$ with some extra conditions satisfied by the exceptional vertices.

Observe that the balanced join of two cycles of lengths $\approx\frac{n}{2}$ has $\approx \frac{n^2}{4}+n$ edges (and joins of cycles of unbalanced lengths have even less edges), so our graph $G$ is only allowed to ``lose'' $\approx \frac{n}{2}$ edges with respect to that number before it violates the bound of Theorem~\ref{lem:extremalgraphtheory}. In many cases, however, we will be able to show that a $2$-joinlike graph loses a lot more just by counting the edges missing in the sparse planar links of exceptional vertices (Definition~\ref{def:fascinating}d)).

This leaves us with just a handful of possible scenarios considered in Section~\ref{section:exact}. Those are the difficult ones, in the sense that the graphs $G$ approach, and in fact even reach, the bound $m=\frac{1}{4}(n^2+2n+17)$. That means we can no longer use rough estimates.
We then have to examine the structure of $G$ more closely.
This is the part where the examples advertised in Remark~\ref{remark:funnygraphs} show up.

\bigskip
Let $e(H)=|E(H)|$ and we write $e(H[A,B])$, (resp.
$\overline{e}(H[A,B])$) for the number of edges (resp. non-edges) crossing between two disjoint vertex sets $A,B\subset V(H)$ .

Let us now state a theorem of Erd\H{o}s and Simonovits~\cite[Theorem~3]{ErdSim:Supersaturated}, tailored to our needs.\footnote{These days, similar theorems are typically proven with the help of the Szemer\'edi Regularity Lemma~\cite{Sze78}; see for example~\cite[Theorem~2.9]{KS96}. Even though the Regularity Lemma was already alive by the time of publishing~\cite{ErdSim:Supersaturated} the theory was too juvenile to yield such a statement back then. Therefore some alternative, ``sieve'' arguments were used instead.} As said above, this version of the Supersaturation Theorem gives an approximate structure in graphs with edge density at least $\frac12$ which contain subcubically many triangles in the order of the graph.\footnote{The general version of the Supersaturation Theorem deals with (hyper)graphs containing a small number of copies of a fixed (hyper)graph $F$.} 
\begin{theorem}\label{thm:stability}
For every $\varepsilon>0$ there exists $\delta>0$ such that the following holds.
Let $H$ be an $h$-vertex graph with at least $h^2/4$ edges, containing at most
$\delta h^3$ triangles. Then there exists a partition $V(H)=A_1\sqcup A_2$, with 
$\big| |A_1|-|A_2|\big|\le 1$, such that 
\begin{equation}\label{eq:suchthat}
e(H[A_1])+e(H[A_2])+\overline{e}(H[A_1,A_2])\le \varepsilon h^2\;.
\end{equation}
\end{theorem}
To obtain the above statement set $\mathcal{L}$ to the one-element family consisting of just a triangle in \cite[Theorem~3]{ErdSim:Supersaturated}.

\medskip
We can now proceed with the proof of Theorem~\ref{lem:extremalgraphtheory}.
Let $0<\gamma\ll 1$, $\alpha<\gamma/1000$ and $\varepsilon<\alpha\gamma$ be fixed. Let $\delta$ be given by Theorem~\ref{thm:stability} for input parameter $\varepsilon$. Let $n_0$ be sufficiently large. Suppose that $G$ is the graph as in Theorem~\ref{lem:extremalgraphtheory}. Definition~\ref{def:fascinating}a) gives us that $G$ has  $2(e(G)-n)<n^2<\delta n^3$ triangles. Therefore, Theorem~\ref{thm:stability} applies with parameters $\delta$ and $\varepsilon$. Let $A_1\sqcup A_2$ be the partition of $V(G)$ from Theorem~\ref{thm:stability}. 

Let us fix additional notation. Given a vertex $v$ and a set of vertices $X$ we write
$$\deg(v,X)=|N_v\cap X|.$$

Define the following vertex sets for $i=1,2$:
\begin{eqnarray*}
B_i &=& \{v\in A_i:\ \deg(v,A_{3-i})\geq \frac{n}{2}-\gamma n\},\\
W_i &=& \{v\in A_i\setminus B_i:\ \deg(v,B_i)\geq \frac{n}{2}-\gamma n\},\\
X_i &=& (A_i\setminus B_i) \setminus W_i.
\end{eqnarray*}

\begin{claim}\label{claim:ci-small}
We have $|A_i\setminus B_i|\leq \alpha n$ for $i=1,2$. In particular $|W_i|,|X_i|<\alpha n$ and $|B_i|\geq \frac{n}{2}-\alpha n$.
\end{claim}
\begin{proof}
By definition every vertex of $A_i\setminus B_i$ has at least $\gamma n-1$ non-edges to $A_{3-i}$. If we had $|A_i\setminus B_i|>\alpha n$ then
$$\bar e(G[A_1,A_2])\geq |A_i\setminus B_i|\cdot (\gamma n-1)\geq \alpha\gamma n^2-\alpha n>\varepsilon n^2\;,$$
contrary to the choice of $A_1$ and $A_2$.
\end{proof}

Now define the partition $V(G)=S_1\sqcup S_2\sqcup X$ as follows
\begin{eqnarray*}
S_i &=& B_i\cup W_{3-i},\\
X &=& X_1\cup X_2.
\end{eqnarray*}

Observe that $\frac{n}{2}-\alpha n\leq |S_i|\leq \frac{n}{2}+\alpha n$ and $|X|\leq 2\alpha n$. Denote $x=|X|$. It is our goal to show that $X=\emptyset$, that $S_1$ and $S_2$ induce cycles, and that the bipartite graph between $S_1$ and $S_2$ is complete.

\begin{claim}\label{claim:si-large-outdeg}
For $i=1,2$ and for every vertex $v\in S_i$ we have $\deg(v,S_{3-i})\geq \frac{n}{2}-2\gamma n$.
\end{claim}
\begin{proof}
If $v\in B_i$ then $v$ has at least $\frac{n}{2}-\gamma n$ neighbors in $A_{3-i}$ and by Claim~\ref{claim:ci-small} at least $\frac{n}{2}-2\gamma n$ of them hit $B_{3-i}$. If $v\in W_{3-i}$ then $v$ has at least $\frac{n}{2}-\gamma n$ neighbors in $B_{3-i}$. 
\end{proof}

\begin{claim}\label{claim:si-degree2}
For $i=1,2$ and for every vertex $v\in S_i$ we have $\deg(v,S_i)\leq 2$. Consequently, $e(G[S_1])+e(G[S_2])\le n$. Moreover, $G[S_i]$ is triangle-free.
\end{claim}
\begin{proof}
Suppose a vertex $v\in S_i$ has three neighbors $u_1,u_2,u_3\in S_i$. By Claim~\ref{claim:si-large-outdeg} we have
$$|N_v\cap N_{u_1}\cap N_{u_2}\cap N_{u_3}\cap S_{3-i}|\geq \frac{n}{2}-13\gamma n\ge 3\;.$$
This implies that $\lk_Gv$ contains a copy of $K_{3,3}$ (with $u_1, u_2, u_3$ on one side and the other being in $S_{3-i}$), which is a contradiction to Definition~\ref{def:fascinating}d).

The proof of the last statement is similar: if $t$ is a triangle in $G[S_i]$ then $\lk_Gt$ contains most of $S_{3-i}$, so $G$ fails Definition~\ref{def:fascinating}c). 
\end{proof}

\begin{claim}\label{claim:x-degree-bound}
If $v\in X$ then $\deg(v,S_i)\leq \frac{n}{2}-\frac{2}{3}\gamma n$ for $i=1,2$.
\end{claim}
\begin{proof}
By definition every vertex $v\in X$ satisfies $\deg(v,B_i)\leq \frac{n}{2}-\gamma n$ for $i=1,2$. Therefore
$$\deg(v,S_i)\leq \deg(v,B_i)+|W_{3-i}|\leq \frac{n}{2}-\gamma n+\alpha n\leq \frac{n}{2}-\frac{2}{3}\gamma n.$$ 
\end{proof}

We call a vertex $v\in X$ \emph{poor} if $\deg(v,S_1)\geq 3$ and $\deg(v,S_2)\geq 3$. Let $P\subset X$ be the set of poor vertices. Choose a partition $X\setminus P=T_1\sqcup T_2$ such that the vertices $v\in T_i$ satisfy $\deg(v,S_i)\leq 2$ for $i=1,2$. Let $p=|P|$. 

\begin{claim}\label{claim:nonpoor-degree}
If $v\in X\setminus P$ then $\deg(v,S_1\cup S_2)\leq \frac{n}{2}-\frac{1}{2}\gamma n$.
\end{claim}
\begin{proof}
This is obvious from Claim~\ref{claim:x-degree-bound}.
\end{proof}

\begin{claim}\label{claim:poor-degree}
If $v\in P$ then $\deg(v,S_i)\leq 12\gamma n$ for $i=1,2$.
\end{claim}
\begin{proof}
Suppose the contrary and without loss of generality let $\deg(v,S_2)>12\gamma n$. Let $u_1,u_2,u_3\in N_v\cap S_1$ be three different vertices. By Claim~\ref{claim:si-large-outdeg} the set $N_{u_1}\cap N_{u_2}\cap N_{u_3}\cap S_{2}$ has at least $\frac{n}{2}-10\gamma n$ vertices, therefore $N_v$ hits at least $\gamma n$ of them. In particular $G[N_v]$ contains a $K_{3,3}$, a contradiction.
\end{proof}

We can now plug in the bounds from the claims above to count the number of edges in $G$ to obtain the following bound
\begin{align*}\begin{split}
\frac{1}{4}n^2+\frac{1}{2}n+\frac{17}{4}<e(G)&\leq e(G[S_1,S_2])+e(G[S_1])+e(G[S_2])+e(G[P,S_1\cup S_2])+\\
&\qquad\qquad\qquad\qquad +e(G[X\setminus P,S_1\cup S_2])+{|X|\choose 2}\\
&\leq \left(\frac{n-x}{2}\right)^2+n+24p\gamma n+(x-p)\left(\frac{n}{2}-\frac{1}{2}\gamma n\right)+\frac{x^2}{2}\;.
\end{split}\end{align*}
This is equivalent to
$$x\left(\frac{\gamma
n}{2}-\frac{3}{4}x\right)+\frac{pn}{2}(1-49\gamma)+\frac{17}{4}<\frac{n}{2}.$$
Since $x\leq 2\alpha n<\frac{1}{3}\gamma n$, we have $\frac{\gamma n}{2}-\frac{3}{4}x>\frac{\gamma n}{4}$, and the last inequality implies
\begin{equation}
\frac{x\gamma n}{4}+\frac{pn}{2}(1-49\gamma)+\frac{17}{4}<\frac{n}{2}.
\end{equation}
It follows that 
\begin{align}
x&<\frac{2}{\gamma}\;\mbox{, and}\label{eq:xABS}\\
p&<\frac{1}{1-49\gamma}<1.5\;.\label{eq:pABS}
\end{align}
In particular we can only have $p=0$ or $p=1$.

\medskip
Let $K_i=S_i\cup T_i$ for $i=1,2$. Note that 
$$\frac{n}{2}-\alpha n\leq |K_i|\leq\frac{n}{2}+\alpha n+x\leq\frac{n}{2}+2\alpha n.$$
Let $b=\overline{e}(G[K_1,K_2])$ be the number of
missing edges between $K_1$ and $K_2$. The following bound follows directly from
Claim~\ref{claim:si-degree2}, the definition of $T_i$ and \eqref{eq:xABS}. 
\begin{claim}\label{claim:degKi}
For each $v\in K_i$ we have that $\deg(v,K_i)\le |T_i|+2\le x+2\leq \frac{4}{\gamma}$.
\end{claim}

\begin{claim}\label{claim:edges}
For $i=1,2$ and each set $Y\subset S_i$, $|Y|\le\frac n8$ we have that $G[S_i\setminus Y]$ contains at least one edge. In particular $G[S_i]$ contains at least one edge.
\end{claim}
\begin{proof}
Suppose the claim does not hold for example for $i=1$ and some set $Y\subset S_1$. Let $t_i$ be the number of triangles in $G$ with at least two vertices in $K_i$.

If $T_2\neq\emptyset$ then let us consider an arbitrary fixed vertex $v\in T_2$. By Claim~\ref{claim:degKi} inside $K_2$ there are at most $\deg(v,K_2)^2\le 16/\gamma^2$ triangles touching $v$. We further see that there are at most $\deg(v,K_2)|K_1|\le 4n/\gamma$ triangles through $v$ with two vertices in $K_2$ (one of them being $v$) and one vertex in $K_1$. Summing over all $v\in T_2$ we get that the number of triangles touching $T_2$ with at least two vertices in $K_2$ is at most $|T_2|\times (\frac{16}{\gamma^2}+\frac{4n}\gamma)\le \frac{17n}{\gamma^2}$.

To bound $t_2$ it only remains to add triangles whose two vertices are in $S_2$ and the third is in $K_1$ (by Claim~\ref{claim:si-degree2} there are no triangles entirely inside $S_2$). By Claim~\ref{claim:si-degree2} we have 
\begin{equation}\label{eq:toStre}
 e(G[S_2])\le |S_2|\le\frac{11 n}{20}\;.
\end{equation}
Since each edge in $S_2$ can be extended in at most $|K_1|\le\frac{11 n}{20}$ ways to such a triangle we get that
$$t_2\le \frac{17 n}{\gamma^2}+\frac{11 n}{20}\cdot \frac{11 n}{20}\le \frac{122 n^2}{400}\;.$$

To bound the number $t_1$ of triangles with at least two vertices inside $K_1$
we proceed similarly, except that the fact $e(G[S_1\setminus Y])=0$ allows us to strengthen the counterpart of~\eqref{eq:toStre} to $e(G[S_1])\le 2|Y|\leq\frac{n}{4}$. Consequently,
$$t_1\le \frac{17 n}{\gamma^2}+\frac{n}{4}\cdot \frac{11 n}{20}\le
\frac{3n^2}{20}\;.$$

Finally the number $t_P$ of triangles passing through the (at most one) poor vertex in $P$ satisfies $t_P\leq (24\gamma n+x)^2<700\gamma^2n^2<0.01n^2$ by Claim~\ref{claim:poor-degree}. 

We get that the total number of triangles is $t_1+t_2+t_P<0.47n^2<2(e(G)-n)$, a contradiction
to Definition~\ref{def:fascinating}a).
\end{proof}

Next, we claim that there are no poor vertices.
\begin{claim}
\label{claim:p=0}
We have $p=0$.
\end{claim}
  \begin{proof}
Suppose that $p=1$ and let $P=\{q\}$. Employing Claim~\ref{claim:si-degree2} and the definition of $T_1, T_2$ we get
\begin{eqnarray*}
e(G[K_1\cup K_2])&\le&
\left(\frac{n-1}2\right)^2-b+\sum_{i=1,2}\left(|S_i|+2|T_i|+{|T_i|\choose 2}\right)\\
&\overset{\eqref{eq:xABS}}\le&
\left(\frac{n-1}2\right)^2-b+n+C\;,
\end{eqnarray*}
where $C$ depends only on $\gamma$. By Claim~\ref{claim:poor-degree} we then have the following
estimate
\begin{equation}\label{ILikeCheese}
\frac{1}{4}n^2+\frac{1}{2}n+\frac{17}{4}<e(G)\leq
\left(\frac{n-1}{2}\right)^2-b+n+25\gamma n.
\end{equation}
This implies
\begin{equation}
\label{eq:lowm}
b\leq 25\gamma n.
\end{equation}
Consider any edge $e\in G[S_1]$. The link $\lk_Ge$ is a cycle $C$ which contains, by Claim~\ref{claim:si-large-outdeg}, at least $\frac{n}{2}-6\gamma n$ vertices of $S_2$ and, by Claim~\ref{claim:si-degree2}, does not pass through $S_1$. The number of vertices in which $C$ can exit $S_2$ is bounded from above by $2(x+1)$. Eliminating the vertices of $C$ which are adjacent (in the graph $G$) to $T_2$ (at most $2x$) or to $q$ (at most $12\gamma n$ by Claim~\ref{claim:poor-degree}) we find that $G[S_2]$ contains at least $\frac{1}{2}(\frac{n}{2}-30\gamma n)$ vertex-disjoint edges $e'=u'v'$ which satisfy $V(\lk_Ge')\subset K_1$. 

We claim that for at least one such edge $e'=u'v'$ we have $K_1\subset N_{u'}\cap N_{v'}$. Indeed, each edge $e'$ for which this does not hold is incident with at least one non-edge in $G[K_1,K_2]$, and thus otherwise we would get at least $\frac{1}{2}(\frac{n}{2}-30\gamma n)$ non-edges in $G[K_1,K_2]$, a contradiction to~\eqref{eq:lowm}.

Let us fix an edge $e'$ as above. We now have that $\lk_Ge'=G[K_1]$ and therefore $G[K_1]$ is a cycle. A symmetric argument starting with an appropriate edge $e''\in G[K_1]$ for which $\lk_Ge''=G[K_2]$ shows that $G[K_2]$ is a cycle as well. 

We now see that $G$, with the decomposition $V(G)=K_1\sqcup K_2\sqcup \{q\}$, is $1$-joinlike in the sense of Definition~\ref{def:joinlike}. We shall however later in Proposition~\ref{prop:1-join-deg-3} show that this leads to a contradiction.
\end{proof}

For the remaining part we can therefore assume $P=\emptyset$. Our short-term
goal for now is to prove that $G$ is $0$-, $1$- or $2$-joinlike. The same way we
derived~\eqref{ILikeCheese} we get that
$$\frac{1}{4}n^2+\frac{1}{2}n+\frac{17}{4}<e(G)\leq \left(\frac{n}{2}\right)^2-b+n+\frac{4}{\gamma^2}.$$ This implies
\begin{equation}
\label{eq:mbound}
b< \frac{n}{2}+\frac{4}{\gamma^2}-\frac{17}4<0.51n.
\end{equation}

Let $E_i$ be the set containing $T_i$ and all the neighbors in $S_i$ of the vertices in $T_i$. By definition of $T_i$ we have $|E_i|\leq 3x$. Note that $K_i\setminus E_i=S_i\setminus E_i$ and for any vertex $v\in K_i\setminus E_i$ we have $\deg(v,K_i)\leq 2$.

Fix two edges $e_1\in G[S_1\setminus E_1]$ and $e_2\in G[S_2\setminus E_2]$; such edges exist by Claim~\ref{claim:edges}. For each $i=1,2$ the link $\lk_G e_{3-i}$ lies in $K_i$ and its intersection with $K_i\setminus E_i$ is a collection of at most $3x$ paths of total length at least $\frac{n}{2}-6\gamma n$ by Claim~\ref{claim:si-large-outdeg}, or a sole cycle. Define a \emph{segment} in $G[K_i]$ as a maximal connected sub-path (or a cycle) of $\lk_G e_{3-i}$ which lies in $K_i\setminus E_i$. (Note that our definition of segments is with respect to fixed edges $e_1$ and $e_2$.) There are at most $3x\leq 6/\gamma$ segments in $K_i$. A segment is called \emph{long} if it has at least $\alpha n$ vertices and \emph{short} otherwise. The total length of short segments in $K_i$ is at most $\frac{6}{\gamma}\cdot\alpha n<0.09 n$, hence the total length of long segments in each $K_i$ is at least $0.4 n$.

\begin{claim}\label{claim:joined-segments}
Let $R_1$ and $R_2$ be two segments in $K_1$ and $K_2$, respectively. If for
some vertices $x_1\in R_1$, $x_2\in R_2$ we have $x_1x_2\in E(G)$
then $G[R_1,R_2]$ is complete bipartite.
\end{claim}
\begin{proof}
If $x_1',x_1''$ are the neighbours of $x_1$ in $K_1$ and $x_2',x_2''$ are the neighbours of $x_2$ in $K_2$, then the link $\lk_Gx_1x_2$ is a cycle contained  in $\{x_1',x_1'',x_2',x_2''\}$, hence, by Definition~\ref{def:fascinating}b) it must pass through all those vertices.  Therefore $x_1x_2',x_1x_2'',x_2x_1',x_2x_1''\in E(G)$. By successively repeating the same argument for the newly forced edges we prove the claim.
\end{proof}

\begin{claim}\label{claim:long-segments}
If $R_1$ and $R_2$ are two long segments in $K_1$ and $K_2$ respectively then
$G[R_1,R_2]$ is complete bipartite.
\end{claim}
\begin{proof}
If not then, by Claim~\ref{claim:joined-segments}, the bipartite graph $G[R_1,R_2]$ does not contain any edges. 
Then
$$\overline e(G[K_1,K_2])\geq \overline e(G[R_1,R_2])=
|R_1|\cdot |R_2|\geq \frac{\alpha^2n^2}2\;,$$ a contradiction
to~\eqref{eq:mbound}.
\end{proof}

Let $L_1$, and $L_2$ be the vertex sets of all the long segments in $K_1$
and $K_2$, respectively. By Claim~\ref{claim:long-segments} the graph $G[L_1,L_2]$ is
complete bipartite. For $i=1,2$ choose edges $\widetilde e_i\in G[L_i]$ which
minimize the quantity
\begin{equation}|K_{3-i}\setminus V(\lk_G
\widetilde e_i)|\label{minimize}\end{equation} and let $C_i\subset K_i$ be the
vertex set of the cycle $\lk_G \widetilde e_{3-i}$.

\begin{claim}\label{claim:at-most-2} We have 
$|K_1\setminus C_1|+|K_2\setminus C_2|\leq 2.$
\end{claim}

\begin{proof}
Let $d_i=|K_i\setminus C_i|$. By the optimality of the choice of $\widetilde e_i$ we
get that the link of every edge in $G[L_i]$ misses at least $d_{3-i}$ vertices of
$K_{3-i}$. Since $G[L_1,L_2]$ is complete bipartite by
Claim~\ref{claim:long-segments}, those missing edges must contribute to
$\overline{e}(G[L_i,K_{3-i}\setminus L_{3-i}])$. Recall that $G[L_i]$ is a collection of at most $3x\leq 6/\gamma$ vertex-disjoint paths (or a cycle) of total
length at least $0.4n$. We get $$\overline e(G[L_i,K_{3-i}\setminus
L_{3-i}])\geq \frac{d_{3-i}}{2}(|L_i|-3x)\geq d_{3-i}\cdot 0.19\cdot n.$$ 
The two sets of missing edges we count this way for $i=1,2$ are disjoint.
Therefore, using~\eqref{eq:mbound}
$$0.51n> b\geq \overline e(G[L_1,K_{2}\setminus L_{2}])+\overline e(G[L_2,K_{1}\setminus L_{1}])\geq 0.19n(d_1+d_2)$$ which implies $d_1+d_2<2.7$. That ends the proof.
\end{proof}

\bigskip
The graphs $G[C_1]$, $G[C_2]$ are cycles and the minimizing edges $\widetilde e_i\in
L_i\subset C_i$ satisfy $\lk_G\widetilde e_i=G[C_{3-i}]$. Together with
Claim~\ref{claim:at-most-2} it shows that $G$ is $t$-joinlike for $t\leq 2$.  If
$t=0$ then we are done. The case $t=1$ leads to a contradiction as shown in
Proposition~\ref{prop:1-join-deg-3}. We can therefore assume that $t=2$
and call the two exceptional vertices $q$ and $q'$. We can assume without
loss of generality that either
\begin{equation}\label{case11}K_1\setminus C_1=\{q\}, \quad K_2\setminus
C_2=\{q'\},\end{equation} or
\begin{equation}\label{case20}K_1\setminus C_1=\{q,q'\}, \quad K_2\setminus
C_2=\emptyset.\end{equation}

Define the following quantities for $i=1,2$,
\begin{align*}
d_i(q)=\deg(q,C_i)\quad&\mbox{and}\quad d_i(q')=\deg(q',C_i)\;,\\
e_i(q)=e(G[N_{q}\cap C_i])\quad&\mbox{and}\quad e_i(q')=e(G[N_{q'}\cap C_i])\;.
\end{align*}

Note that $e_i(q)\leq d_i(q)$ and $e_i(q')\leq d_i(q')$ since $G[N_{q}\cap C_i]$ and $G[N_{q'}\cap C_i]$ are induced subgraphs of cycles.

If any of the numbers $d_1(q), d_1(q'), d_2(q), d_2(q')$ is
at most~2, then the result follows from Proposition~\ref{prop:2-join-all}. We will therefore assume that $$\min\{d_1(q), d_1(q'), d_2(q), d_2(q')\}\ge 3\;.$$
The proof under this assumption splits into the two cases \eqref{case11} and \eqref{case20} and is presented in the next section.

\section{Two exceptional vertices of large degrees}
\label{section:2large}
In this section we show that each of the cases~\eqref{case11} and~\eqref{case20}
from the previous section leads to a contradiction. We use the same notation.

We are going to exploit the fact that the graphs $\lk_G q$ and $\lk_G q'$ are
planar. Recall that Euler's formula implies  an $h$-vertex planar graph can have
at most $3h-6$ edges. So, planar graphs are sparse,
and a substantial number of edges must be missing between $C_1$ and
$C_2$. A careful edge counting will lead to a contradiction.

We start with an auxiliary claim.
\begin{claim}
\label{claim:planaredgecount}
We have an inequality
$$\overline{e}(G[N_{q}\cap C_1, N_{q}\cap C_2])\geq
d_1(q)d_2(q)-3d_1(q)-3d_2(q)+e_1(q)+e_2(q)+6\;.$$
An analogous inequality holds for $q'$.
\end{claim}
\begin{proof}
The graph $G[N_{q}\cap(C_1 \cup C_2)]$ is a planar graph with $d_1(q)+d_2(q)$
vertices and $$d_1(q)d_2(q)-\overline{e}(G[N_{q}\cap C_1, N_{q}\cap
C_2])+e_1(q)+e_2(q)$$ edges. The claim now follows from Euler's formula. 
\end{proof}

From previous estimates we have $\frac{n}{2}-2\alpha n\leq |C_i|\leq
\frac{n}{2}+2\alpha n$. The next easy statement records the fact that if $q$ is
adjacent to most of $C_i$ then $\lk_Gq$ also contains most of the edges from
$G[C_i]$.
\begin{claim}
\label{easy-pizi}
Suppose $\beta\geq 4\alpha$. If $d_i(q)\geq \frac{n}{2}(1-\beta)$ then
$e_i(q)\ge \frac{n}{2}(1-5\beta)$. The same holds for $q'$.
\end{claim}
\begin{proof}
Since $|C_i|\leq \frac{n}{2}+2\alpha n$ the set $N_{q}$ misses at most
$$\frac{n}{2}+2\alpha n-\frac{n}{2}(1-\beta) = n\left(\frac12\beta
+2\alpha\right)\leq \beta n$$ vertices of $C_i$. Recall that $G[C_i]$ is a
cycle. It follows that at most $2\beta n$ edges
of $G[C_i]$ are not in $\lk_Gq$. Hence $$e_i(q)\geq \frac{n}{2}-2\alpha n-2\beta n=\frac{n}{2}(1-4\alpha-4\beta)\geq \frac{n}{2}(1-5\beta).$$
\end{proof}


\medskip
\subsection{The case \eqref{case11}.}
By Claim~\ref{claim:degKi} we have $d_1(q),d_2(q')\leq \frac4\gamma$. Therefore $$\overline{e}(G[C_1,C_2])\geq
\overline{e}(G[N_q\cap C_1,N_q\cap C_2])+\overline{e}(G[N_{q'}\cap C_1,N_{q'}\cap C_2])-\frac{16}{\gamma^2}\;.$$

The inequality
\begin{align*}
\frac{1}{4}(n^2+2n+17)&< e(G)\leq
\left(\frac{n-2}{2}\right)^2+n+\deg(q)+\deg(q')-\bar{e}(G[C_1,C_2])\\ &\leq
\frac{n^2}{4}+\deg(q)+\deg(q')\\
&~~~~-\bar{e}(G[N_q\cap C_1,N_q\cap
C_2])-\bar{e}(G[N_{q'}\cap C_1,N_{q'}\cap C_2])+\frac{16}{\gamma^2}+1
\end{align*}
together with Claim~\ref{claim:planaredgecount} and $e_1(q),e_2(q')\le \frac{4}{\gamma}$ gives
\begin{equation}
\label{local1}
\frac12 n+(d_1(q)-4)(d_2(q)-4)+(d_1(q')-4)(d_2(q')-4)+e_2(q)+e_1(q')\leq O(1)\;,
\end{equation}
where $O(1)$ denotes some universal constant (depending on $\gamma$) whose exact value does not matter.
Observe that if $d_1(q)\geq 4$ then the inequalities $d_1(q)\leq \frac{4}\gamma$
and $d_2(q)\geq 3$ imply $(d_1(q)-4)(d_2(q)-4)\geq -\frac{4}\gamma$. A similar
observation holds for $q'$. Therefore, if $d_1(q),d_2(q')\geq 4$ then
we get a contradiction because then the left-hand  side of~\eqref{local1} is at
least $\frac12 n-\frac{8}\gamma$.

Let us then assume that $d_1(q)=3$. Then the inequality \eqref{local1} becomes
\begin{equation}
\label{knedliki}
\frac12 n+(d_1(q')-4)(d_2(q')-4)+e_2(q)+e_1(q')\leq d_2(q)+O(1)\;.
\end{equation}
If $d_2(q')\geq 4$ then $(d_1(q')-4)(d_2(q')-4)\ge -\frac{4}{\gamma}$, and
therefore~\eqref{knedliki} implies $d_2(q)\geq 0.49n$. By Claim~\ref{easy-pizi} we
have $e_2(q)\geq 0.45n$ and plugging this back into \eqref{knedliki} we get
$d_2(q)\geq \frac12 n+0.45n-O(1)\geq 0.94n$, which is a contradiction with
$d_2(q)\leq |C_2|\leq 0.51n$.

We are now left with the case when $d_1(q)=d_2(q')=3$ and~\eqref{knedliki}
reduces to
\begin{equation}
\label{eq:opp-sides-deg3}
\frac12 n+e_2(q)+e_1(q')\leq d_2(q)+d_1(q')+O(1)\;.
\end{equation}
We now need the following claim.

\begin{claim}
If $v\in C_2$ is an isolated vertex of the graph $G[N_q\cap C_2]$ then $vq'\in
E(G)$.
\end{claim}
\begin{proof}
The cycle $\lk_Gqv$ is contained in $(N_q\cap C_1)\cup\{q'\}$ and since
$d_1(q)=3$, the latter set has $4$ vertices. By
Definition~\ref{def:fascinating}b) $\lk_Gqv$ must pass through all of them
and in particular $q'\in N_v$.
\end{proof}
Because $d_2(q')=3$ the claim implies that $G[N_q\cap C_2]$ can have at most $3$
isolated vertices and therefore $e_2(q)\geq \frac12 (d_2(q)-3)$. By symmetry we
get  $e_1(q')\geq \frac12 (d_1(q')-3)$ and~\eqref{eq:opp-sides-deg3} implies
\begin{equation}
n\leq d_1(q')+d_2(q)+O(1)\;.
\end{equation}
It follows that $d_1(q'),d_2(q)\geq 0.48 n$ but then, by Claim~\ref{easy-pizi},
$e_1(q'),e_2(q)\geq 0.4n$ and going back to the inequality \eqref{eq:opp-sides-deg3} gives a contradiction.

\medskip
\subsection{The case \eqref{case20}.}

This time we have $d_1(q),d_1(q')\leq \frac{4}\gamma$. The missing edges in
$G[N_q\cap C_1, N_q\cap C_2]$ and $G[N_{q'}\cap C_1, N_{q'}\cap C_2]$ can have a
significant overlap, so we begin by using just the contribution of one of them to obtain a bound. We have
\begin{equation*}
\frac{1}{4}(n^2+2n+17)< e(G)\leq
\left(\frac{n-2}{2}\right)^2+n+\deg(q)+\deg(q')-\bar{e}(G[N_q\cap C_1,N_q\cap C_2])\;,
\end{equation*}
and plugging in the bound from Claim~\ref{claim:planaredgecount} we obtain
\begin{equation}
\label{eq:p}
\frac12 n+(d_1(q)-4)(d_2(q)-4)+e_2(q)\leq d_2(q')+O(1)\;.
\end{equation}
In the same way we obtain a symmetric version with $q$ and $q'$ interchanged:
\begin{equation}
\label{eq:q}
\frac12 n+(d_1(q')-4)(d_2(q')-4)+e_2(q')\leq d_2(q)+O(1)\;.
\end{equation}
Now suppose that $d_1(q)\geq 4$. Then $(d_1(q)-4)(d_2(q)-4)\geq -\frac4\gamma$,
and so~\eqref{eq:p} implies $d_2(q')\geq 0.49 n$. Therefore, $e_2(q')\geq
0.45n$ by Claim~\ref{easy-pizi}. Then the inequality~\eqref{eq:q} can be rewritten as
\begin{align*}
d_2(q)& \geq  \frac12 n+(d_1(q')-4)(d_2(q')-4)+e_2(q')-O(1)\\
& \geq 0.94n+(d_1(q')-4)(d_2(q')-4)\;.
\end{align*}
This inequality can only be satisfied if the last product is negative, which
implies $d_1(q')=3$. Using $d_2(q')\leq 0.51n$ we further obtain
\begin{align*}
d_2(q)& \geq  0.94n-0.51n=0.43n\;.
\end{align*}
By Claim~\ref{easy-pizi} we get $e_2(q)\geq 0.15n$, but then~\eqref{eq:p} gives 
$$d_2(q')\geq \frac12 n+0.15n-O(1)\geq 0.64n\;,$$
which is a contradiction.

By symmetry we also arrive at a contradiction assuming that $d_1(q')\geq 4$. It means that we must have $d_1(q)=d_1(q')=3$.

We have that $|(N_q\cup N_{q'})\cap C_1|\le 6$. Consequently, there are only a
finite number of possibilities for the graph $G[(N_q\cup N_{q'})\cap C_1]$. We
will first show that the actual possibilities for $G[(N_q\cup N_{q'})\cap C_1]$ are even more limited. Call a vertex $v\in C_1$ \emph{free} if $v\not\in N_q\cup N_{q'}$, a \emph{$q$-vertex} if $v\in N_q\setminus N_{q'}$, a
\emph{$q'$-vertex} if $v\in N_{q'}\setminus N_{q}$, a \emph{$qq'$-vertex} if
$v\in N_q\cap N_{q'}$ and a \emph{boundary} vertex if $v$ belongs to an edge
$e\in G[C_1]$ such that $\lk_G e\cap\{q,q'\}=\emptyset$. Observe that each free vertex is also boundary.

\begin{claim}
\label{claim:pq-stuff}
The vertices in $C_1$ have the following properties: 
\begin{itemize}
\item[a)] if $v\in C_1$ is boundary then $C_2\subset N_v$,
\item[b)] if $v\in C_1$ is a $q$-vertex then at least one of its neighbors in $C_1$ is in $N_q$,
\item[b')] if $v\in C_1$ is a $q'$-vertex then at least one of its neighbors in
$C_1$ is in $N_{q'}$,
\item[c)] if $v\in C_1$ is a $qq'$-vertex then at least one of its neighbors in
$C_1$ is in $N_q\cup N_{q'}$,
\item[d)] if $e_1,e_2\in G[C_1]$ are two vertex-disjoint edges, such that
$\lk_Ge_1$ contains $q$ but not $q'$ and $\lk_Ge_2$ contains $q'$ but not $q$,
then in at least one of those edges both endpoints are non-boundary,
\item[e)] if $v$ is a $q$-vertex and $w$ is a $q'$-vertex then $vw\not\in
E(G[C_1])$.
\end{itemize}
\end{claim}
\begin{proof}
a) Consider any edge $e\in G[C_1]$ such that $v\in e$ and $V(\lk_G e)\cap\{q,q'\}=\emptyset$. Then $\lk_Ge=G[C_2]$, so in particular $C_2\subset N_v$.

b) Let $v',v''\in C_1$ be the neighbors of $v$. If none of $v',v''$ is in $N_q$ then all three of $v,v',v''$ are boundary, so by a) all are adjacent to the whole $C_2$. Pick any vertex $w\in N_q\cap C_2$ and let $w',w''$ be its neighbors in $C_2$. Then the link $\lk_Gvw$ contains the cycle $w'v'w''v''$ and the vertex $q$, which is impossible. By symmetry we also get~b').

c) The proof is the same as b).

d) Suppose the contrary. Let $e_1=xx'$, $e_2=yy'$ where $x'$ and $y'$ are boundary vertices. By a) $C_2\subset N_{x'},N_{y'}$, therefore
$$\lk_G(e_1)=G[\{q\}\cup(N_x\cap C_2)],\quad \lk_G(e_2)=G[\{q'\}\cup(N_y\cap C_2)].$$
It follows that $G[N_x\cap C_2]$ is a path within $C_2$ and $q$ is adjacent only to the endpoints of that path. The same argument for $y$ and $q'$ shows that $G[N_y\cap C_2]$ is a path with $q'$ adjacent only to the endpoints of that path. It follows that, except for up to $4$ special vertices, every vertex in $C_2$ is missing an edge to either $q$ or $x$ and it is missing an edge to either $q'$ or $y$. Since $x,y,q,q'$ are four different vertices this yields at least $2(|C_2|-4)\approx n$ missing edges from $K_2$ to $K_1$, contradicting~\eqref{eq:mbound}.

e) Suppose $vw$ is an edge. Then $v$ and $w$ are both boundary. Let $v'vww'$ be the $4$-vertex path on the cycle $G[C_1]$. By~b) and~b') we have $v'\in N_q$ and $w'\in N_{q'}$. Then the edges $vv'$ and $ww'$ contradict d).
\end{proof}

\begin{center}
\begin{figure}
\includegraphics[scale=0.6]{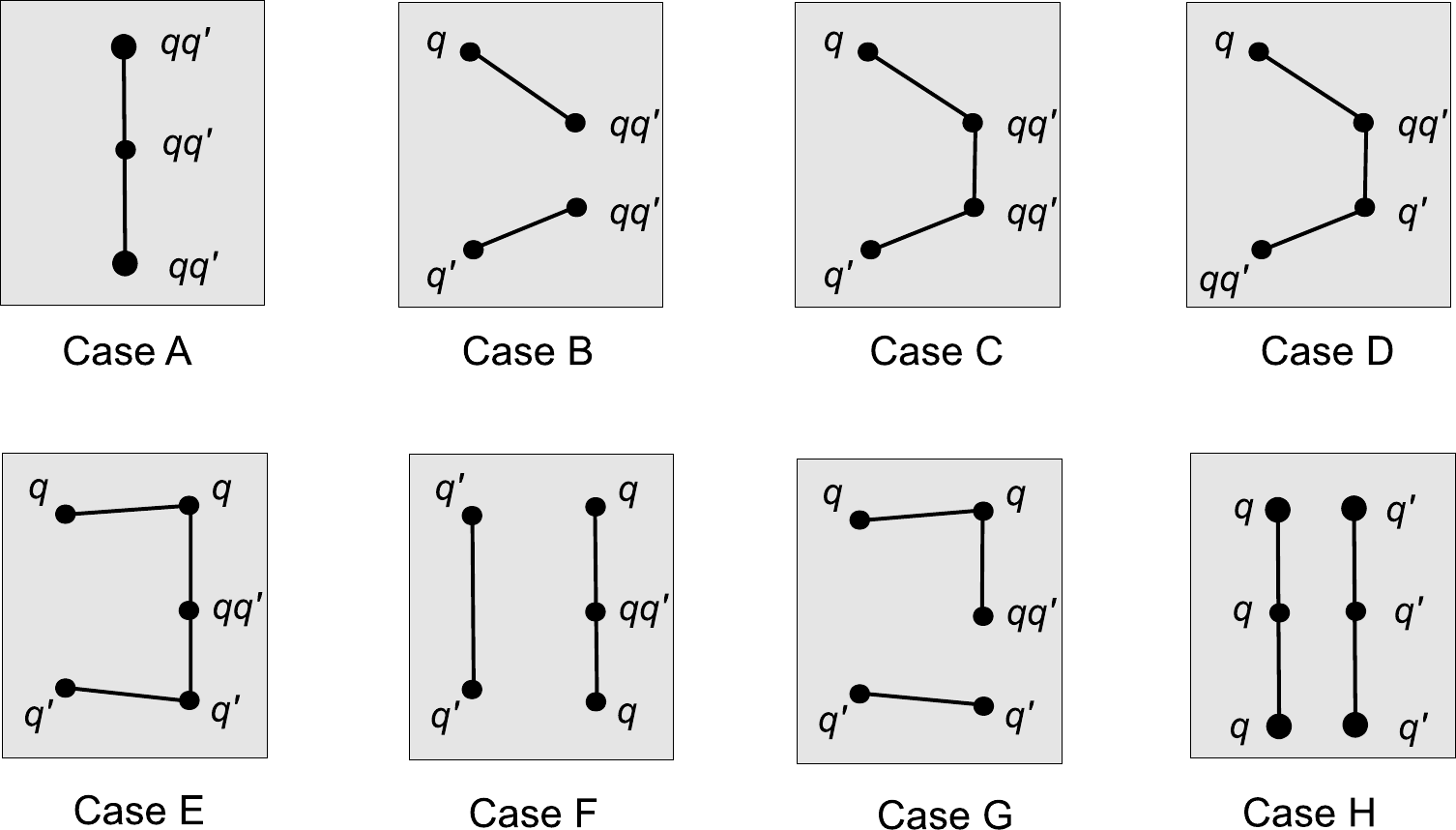} 
\caption{Seven possibilities of the graph $G[(N_q\cup N_{q'})\cap C_1]$ with the
types of the vertices (types $q$-, $q'$-, and $qq'$-).
\label{fig:possibilities}}
\end{figure}
\end{center} 
It turns out that Claim~\ref{claim:pq-stuff} provides enough information to restrict $G[(N_q\cup N_{q'})\cap C_1]$ to just one possibility.
\begin{claim}\label{claim:GqqCdetermined}
We have $N_{q}\cap C_1=N_{q'}\cap C_1=\{v_1,v_2,v_3\}$ where $v_1,v_2,v_3$ are three consecutive vertices in $C_1$.
\end{claim}
\begin{proof}
Claim~\ref{claim:pq-stuff} gives us that $G[(N_q\cup N_{q'})\cap C_1]$ is a graph with no cycle, in which every vertex has degree $1$ or $2$, and there is no edge from a $q$-vertex to a $q'$-vertex. By considering the possible number of $qq'$-vertices ($3$, $2$, $1$ or $0$) and then their degrees, we obtain eight graphs which satisfy the above property, up to exchanging $q$ and $q'$. They are shown in Figure~\ref{fig:possibilities}. The graphs B--H have a pair of edges which violates Claim~\ref{claim:pq-stuff}d). That leaves us only with Case A.
\end{proof}

As all the vertices in $C_1$ except $v_2$ are boundary, we have by Claim~\ref{claim:pq-stuff}a) that $C_2\subset N_{v}$ for each $v\in C_1\setminus \{v_2\}$.
\begin{claim}
\label{claim:nopqedge}
There is no edge $e\in G[C_2]$ with $q,q'\in\lk_G e$.
\end{claim}
\begin{proof}
If $e$ was such an edge then $v_1$ would be a vertex of degree $3$ in $\lk_G e$.
\end{proof}

\begin{claim}
\label{claim:small}
We have $|N_q\cap N_{v_2}\cap C_2|\leq 2$ and $|N_{q'}\cap N_{v_2}\cap C_2|\leq 2$.
\end{claim}
\begin{proof}
Any $3$ vertices in $N_q\cap N_{v_2}\cap C_2$ together with $\{v_1,v_2,v_3\}$ would form a $K_{3,3}$ in $\lk_Gq$, contradicting Definition~\ref{def:fascinating}d).
\end{proof}

To complete the proof we consider two cases. First suppose $qq'\in E(G)$. Then, we have $|N_q\cap N_{q'}\cap C_2|\leq 2$. Indeed, otherwise $v_1$ would be a vertex of degree at least $3$ in $\lk_Gqq'$, a contradiction to Definition~\ref{def:fascinating}b). It follows that every vertex of $C_2$, except for at most $6$ special ones, is adjacent to at most one element of $\{q,q',v_2\}$, and then there at least $2(|C_2|-6)\approx n$ edges missing from $K_2$ to $K_1$. This contradicts~\eqref{eq:mbound}.

Now suppose $qq'\not\in E(G)$. Then $\lk_Gqv_3=G[\{v_2\}\cup(N_q\cap C_2)]$ and $\lk_Gq'v_3=G[\{v_2\}\cup(N_{q'}\cap C_2)]$.
It means that $G[N_q\cap C_2]$ and $G[N_{q'}\cap C_2]$ are paths -- say $P$ and $P'$ -- within $C_2$. By Claim~\ref{claim:nopqedge}, $P$ and $P'$ share at most the endvertices. Moreover, the interior vertices of $P$ and $P'$ are not adjacent to $v_2$. Consequently, every vertex in $C_2$, except for at most $4$ special vertices, is adjacent to at most one element of $\{q,q',v_2\}$. Again, the total number of missing edges from $K_2$ to $K_1$ is at least $2(|C_2|-4)\approx n$, contradicting~\eqref{eq:mbound}.

\smallskip
This ends the consideration of the case \eqref{case20}, thereby completing the proof of Theorem~\ref{lem:extremalgraphtheory}.

\section{Exact results}
\label{section:exact}

In the proof of Theorem~\ref{lem:extremalgraphtheory} we used, as black-boxes, two results about the sparseness of certain $1$- and $2$-joinlike graphs --- Propositions~\ref{prop:1-join-deg-3} and \ref{prop:2-join-all}. They will be proved in this section. Unlike previously, when we were free to count edges with an accuracy of $\Theta(n)$, in this part we will need to determine the precise structure of some fascinating graphs and count their edges exactly.

In this section $G$ means any fascinating graph, which will always be $1$- or $2$-joinlike, with $C_1$, $C_2$ referring to the cycles from Definition~\ref{def:joinlike} and with exceptional vertices called $q$ and $q'$. We will frequently use the observation that if $q$ is an exceptional vertex of a $t$-joinlike graph $G$ then $C_i\setminus N_q\neq\emptyset$ for $i=1,2$.

\begin{proposition}
\label{prop:1-join-deg-012}
If $G$ is $1$-joinlike and $q$ is the exceptional vertex then $\deg(q,C_i)\geq 3$ for $i=1,2$.
\end{proposition}
\begin{proof}
Suppose that $\deg(q,C_1)\leq 2$. If $\deg(q,C_2)=0$ then $\lk_Gq$ contains at most $2$ vertices, so $G$ fails Definition~\ref{def:fascinating}d). Otherwise let $x\in N_q\cap C_2$ be any vertex with at least one neighbor in $C_2\setminus N_q$. We see that $\lk_Gqx$ contains at most $3$ vertices, which is a contradiction.
\end{proof}

\begin{proposition}
\label{prop:1-join-deg-3}
If $G$ is $1$-joinlike then $e(G)\leq \frac{1}{4}(n^2+2n+17)$, where $n=|V(G)|$.
\end{proposition}
\begin{proof}
Let $q$ be the exceptional vertex. We will say that a vertex $v\in C_i$ is a
\emph{$q$-vertex} if $qv\in E(G)$, a \emph{free vertex} otherwise and a
\emph{boundary vertex} if it is a $q$-vertex adjacent to a free vertex.

We refer to $C_1$ and $C_2$ as ``sides''.

\begin{claim}\label{claim:H}
If $v\in C_i$ is free or boundary then $C_{3-i}\subset N_v$.
\end{claim}
\begin{proof}
Indeed, $v$ belongs to an edge $e\in G[C_i]$ with $q\not\in\lk_Ge$ and therefore
with $\lk_Ge=G[C_{3-i}]$. That means $C_{3-i}\subset N_v$.
\end{proof}

By Proposition~\ref{prop:1-join-deg-012} and because $N_q\cap C_i\neq C_i$ for
$i=1,2$, there are at least three $q$-vertices and at least two boundary
vertices on each side. If there were $3$ boundary vertices in, say, $C_1$, then
the graph formed by those $3$ vertices in $C_1$ and any $3$ neighbors of $q$ in
$C_2$ would form, by Claim~\ref{claim:H}, a $K_{3,3}$ in $\lk_Gq$, which is
impossible. That implies there are exactly $2$ boundary vertices on each side.
In other words each $N_q\cap C_i$ induces a path inside $C_i$ of some length
$a_i\geq 3$ for $i=1,2$.

If $u\in C_1$ and $w\in C_2$ are $q$-vertices which are not boundary and $uw\in
E(G)$ then by Claim~\ref{claim:H} there is a $K_{3,3}$ in $\lk_Gq$ formed by
$u$, $w$ and the $2$ boundary vertices on each side. This means $uw\not\in E(G)$ for such $u,w$.

We now know the exact structure of $G$ and we can compute its number of edges. Denoting $c_i=|C_i|$ and using $n=c_1+c_2+1$ we have
\begin{align*}
e(G)&=c_1c_2+c_1+c_2+a_1+a_2-(a_1-2)(a_2-2)\\
&=\frac{1}{4}(n^2+2n+17)-\frac{1}{4}(c_1-c_2)^2-(a_1-3)(a_2-3)\leq \frac{1}{4}(n^2+2n+17).
\end{align*}
\end{proof}

\bigskip
The second part of the analysis in this section deals with 2-joinlike graphs. We
start off by a counterpart of Proposition~\ref{prop:1-join-deg-012}.
\begin{proposition}
\label{prop:2-join-deg-01}
If $G$ is $2$-joinlike and $q$ is any exceptional vertex then $\deg(q,C_i)\geq 2$ for $i=1,2$.
\end{proposition}
\begin{proof}
Suppose that $\deg(q,C_1)\leq 1$. If $\deg(q,C_2)=0$ then $\lk_Gq$ contains at most $2$ vertices, so $G$ fails Definition~\ref{def:fascinating}d). Otherwise let $x\in N_q\cap C_2$ be any vertex with at least one neighbor in $C_2\setminus N_q$. We see that $\lk_Gqx$ contains at most $3$ vertices, which is a contradiction.
\end{proof}

We shall later need the following simple inequality.
\begin{lemma}
\label{stupid-inequality}
If $n=k+l+2$ then
$$kl+2k+l+6\leq \frac{1}{4}(n^2+2n+17).$$
\end{lemma}
\begin{proof}
One checks that
$$kl+2k+l+6=\frac{1}{4}(n^2+2n+17)-\frac{1}{4}(l-k+1)^2.$$
\end{proof}

Proposition~\ref{prop:2-join-all} below is a combination of a case distinction
captured by Proposition~\ref{prop:2-join-deg-2-adjacent} and
Proposition~\ref{prop:2-join-deg-2-non-adjacent}.
\begin{proposition}
\label{prop:2-join-deg-2-adjacent}
If $G$ is $2$-joinlike with exceptional vertices $\{q,q'\}$ such that
$\deg(q,C_1)=2$ and the two vertices of $N_q\cap C_1$ are adjacent, then $e(G)\leq \frac{1}{4}(n^2+2n+17)$, where $n=|V(G)|$.
\end{proposition}
\begin{proof}
Let $N_q\cap C_1=\{u,v\}$. Let $x,x'\in C_2$ be neighbors such that $qx\in E(G)$, $qx'\not\in E(G)$ and let $y$ be the other neighbor of $x$ in $C_2$ (their existence is guaranteed by Proposition~\ref{prop:2-join-deg-01} and the fact that $N_q\cap C_2\neq C_2$). Then $V(\lk_G qx)\subset \{u,v,q',y\}$, and since $uv\in E(G)$ we can assume that $\lk_G qx$ is the cycle $vuyq'$ (this is the unique possibility up to the order of $u, v$). In particular $qq', q'v\in E(G)$ and $q'u\not\in E(G)$. 

If $u'\neq v$ is the other neighbor of $u$ in $C_1$ then $\lk_G uu'$ contains neither $q$ nor $q'$, so it must be all of $C_2$. In particular $C_2\subset N_u$. It means that $\lk_G uq=G[\{v\}\cup (N_q\cap C_2)]$, so $G[N_q\cap C_2]$ is a path of length at least $3$ within $C_2$, whose both endpoints, call them $v_1$, $v_2$, are connected to $v$, while the interior vertices of the path are not connected to $v$. (In fact $x$ from the previous paragraph is one of the $v_i$). Let $a=|N_q\cap C_2|$ be the length of this path.

The link of every edge in $G[N_q\cap C_2]$ contains $u$ and $q$, so to be a cycle it must also contain $q'$. It follows that $N_{q'}\cap C_2\supseteq N_q\cap C_2$.

Let $t\neq u$ be the other neighbor of $v$ in $C_1$. 
We now focus on the link $\lk_Gq'v$. It contains the path $v_1qv_2$. As we shall see, the case $t\not\in \lk_Gq'v$ will lead to a contradiction.
\begin{claim}\label{claim:addedFORREFEREE}
If $t\not\in \lk_Gq'v$ then $\lk_Gq'v$
must contain, apart from $v_1, q$ and $v_2$, all the vertices in
$C_2\setminus N_q$.
\end{claim}
\begin{proof}
The link $\lk_Gq'v$ is a cycle which passes through $v_1qv_2$. The only possible route for this cycle which does not take it outside $\lk_G v$ and avoids $t$ and $u$ is to continue from $v_2$ back to $v_1$ in $C_2$, i.e., follow the path $G[C_2\setminus N_q]$.
\end{proof}
However, the above would imply $C_2\setminus N_q\subset N_{q'}$.
Put together with the previously established $N_{q'}\cap C_2\supseteq N_q\cap C_2$ we
would get $C_2\subset N_{q'}$, a contradiction. This means that
$t\in\lk_Gq'v$, i.e. $q't\in E(G)$.

Consider any vertex $x\in (C_1\cap N_{q'})\setminus \{v\}$ which has at least
one neighbor $\tilde x$ in $C_1\setminus N_{q'}$. By the fact that $q'u\not\in E(G)$ such a
vertex must exist. The link $\lk_G x\tilde x$ is a cycle which does not touch
$C_1\cup\{q,q'\}$. Consequently, $\lk_G x\tilde x=G[C_2]$, and in particular,
$C_2\subset N_x$. The link $\lk_G xq'$ consists of one vertex in $C_1$ and of
the whole $N_{q'}\cap C_2$. We get that $G[N_{q'}\cap C_2]$ is a path within $C_2$,
containing $N_q\cap C_2$. Let $w_1,w_2$ be the endpoints and let $b=|N_{q'}\cap
C_2|$. Assume that $v_1$ is between $w_1$ and $v_2$ on this path (possibly
$w_1=v_1$ or $w_2=v_2$).

For every edge $e$ in $G[(C_2\setminus N_{q'})\cup\{w_1,w_2\}]$ we have
$\lk_Ge=G[C_1]$. As $C_2\cap N_{q'}$ induces a path with endvertices $w_1$ and
$w_2$ and $G[C_2]$ is a cycle, we must have that $G[(C_2\setminus
N_{q'})\cup\{w_1,w_2\}]$ is a path, in particular this graph contains no isolated vertices.
It follows that for every vertex $x\in (C_2\setminus
N_{q'})\cup\{w_1,w_2\}$ we have $C_1\subset N_x$. Now consider the link
$\lk_Gq'v$. It contains the vertices $q, t, v_1, v_2, w_1, w_2$, with paths
$v_1qv_2$ and $w_1tw_2$. This is only possible if $v$ is adjacent to all of
$(N_{q'}\setminus N_q)\cap C_2$ while $t$ is not adjacent to
any vertex of $(((N_{q'}\setminus N_q)\cap C_2)\cup\{v_1,v_2\})\setminus
\{w_1,w_2\}$.

Let $|C_1|=k, |C_2|=l$, with $n=k+l+2$. The remaining part of the proof splits into two cases. First we assume that $t$ is non-adjacent to all of $(N_q\cap C_2)\setminus \{v_1,v_2\}$. In that case $t$ is non-adjacent to $b-2$ vertices of $C_2$, $v$ is non-adjacent to $a-2$ vertices and using a bound $\deg(q',C_1)\leq k-1$ we get
\begin{align*}
e(G)&\leq  kl+k+l+(a+2)+(b+k-1)+1-(a-2)-(b-2)\\
&= kl+2k+l+6\;,
\end{align*}
so the conclusion follows from Lemma~\ref{stupid-inequality}.

Next suppose that $t$ has a neighbor $y$ in $(N_q\cap C_2)\setminus \{v_1,v_2\}$ and let $s\neq v$ be the other neighbor of $t$ in $C_1$. The link $\lk_Gq't$ contains $v,w_1,w_2,y$ and possibly $s$ with edges $w_1vw_2$, and apart from $v$ and $s$ it is contained in $N_{q'}\cap C_2$. Any cycle with that property must contain an edge $e\in G[N_q\cap C_2]$ and it follows that there exists an edge $e\in G[N_t\cap N_q\cap C_2]$. But $\lk_Ge$ is a cycle passing through $uqq't$ and not through $v$, therefore necessarily going through all of $C_1\setminus\{v\}$. In particular $N_{q'}\cap C_1=\{v,t\}$ and so $s\not\in\lk_G q't$. It means that $\lk_Gq't=G[\{v\}\cup(N_{q'}\cap C_2)]$ which, by the restrictions on $N_t$, implies $v_1=w_1$, $v_2=w_2$, $a=b$ and $C_2\subset N_t$. This determines the graph $G$ and we obtain
\begin{align*}
e(G)&=  kl+k+l+(a+2)+(a+2)+1-(a-2)\\
&= kl+k+l+a+7\\
&= \frac{1}{4}(n^2+2n+17)-\frac{1}{4}(k-l+1)^2 -(l-1-a) \leq \frac{1}{4}(n^2+2n+17)
\end{align*}
because $a\leq l-1$.
\end{proof}

\begin{proposition}
\label{prop:2-join-deg-2-non-adjacent}
If $G$ is $2$-joinlike with exceptional vertices $\{q,q'\}$ such that $\deg(q,C_1)=2$ and the two vertices of $N_q\cap C_1$ are not adjacent, then $e(G)\leq \frac{1}{4}(n^2+2n+17)$ where $n=|V(G)|$.
\end{proposition}
\begin{proof}
The proof uses similar techniques as the proof of Proposition~\ref{prop:2-join-deg-2-adjacent}. Set $N_q\cap C_1=\{u,v\}$. 

Let $x\in C_2$ be any vertex with $qx\in E(G)$ and such that $x$ has a neighbor $x'\in C_2$ with $qx'\not\in E(G)$. Let $y$ be the other neighbor of $x$ in $C_2$. We have $V(\lk_Gqx)\subset \{u,v,q',y\}$, with $u$ and $v$ being independent. It follows that $\lk_Gqx$ is the cycle $uq'vy$, in particular $q'u,q'v,ux,vx,q'x\in E(G)$ and $qq'\in E(G)$.

It follows that the number of vertices $x\in C_2$ with the property described in the previous paragraph is at most $2$. Indeed, we proved that every such vertex is adjacent to $u,v,q'$, and the claim follows since $\lk_Gq$ is $K_{3,3}$-free. It means that $G[N_q\cap C_2]$ is a path within $C_2$ of length $a=|N_q\cap C_2|$. Moreover, if $v_1,v_2\in C_2$ are the endpoints of that path then $q'v_j,uv_j,vv_j\in E(G)$ for $j=1,2$. It follows that $a\geq 3$ as otherwise $\lk_Gqu$ would contain a triangle $q'v_1v_2$.

The link $\lk_Gqu$ contains $q',v_1,v_2$ and no vertex in $C_1$, so it must be $G[\{q'\}\cup(N_q\cap C_2)]$. That, and the same argument for $\lk_Gqv$ mean that $N_q\cap C_2\subset N_u, N_v$ and that $q'$ is non-adjacent to vertices in $(N_q\cap C_2)\setminus\{v_1,v_2\}$.

We will now prove the following claim.
\begin{claim}\label{claim:M}
Suppose $x\in C_1\setminus\{u,v\}$ and $y\in (N_q\cap C_2)\setminus\{v_1,v_2\}$. Let $x',x''$ be the neighbors of $x$ in $C_1$, and let $y',y''$ be the neighbors of $y$ in $C_2$. If $xy\in E(G)$ then $xy',xy'',x'y,x''y\in E(G)$.
\end{claim}
\begin{proof}
The link $\lk_Gxy$ contains neither $q$ nor $q'$. Hence it must be contained in $\{x',x'',y',y''\}$, and it follows that these $4$ vertices must form a 4-cycle with $x$ and $y$ adjacent to all of them.
\end{proof}

The vertices $u,v$ divide $G[C_1]$ into two paths which we call $P_1, P_2$, so that there is a partition $C_1=P_1\sqcup P_2\sqcup \{u,v\}$. We also write $\overline{P_j}=P_j\cup\{u,v\}$ for $j=1,2$ for the ``closures'' of those paths. Claim~\ref{claim:M} implies that for $j=1,2$ the bipartite graph $G[P_j,(N_q\cap C_2)\setminus \{v_1,v_2\}]$ is either edgeless or complete bipartite. 
Suppose first that both of these graphs are complete. Take any edge $e$ in $G[N_q\cap C_2]$. As $a\ge 3$, such an edge exists. The above then gives that $\lk_G e$ contains all of $C_1$, and $q$, a contradiction.
Suppose next that both of these graphs are empty. Taking any edge $e$ in $G[N_q\cap C_2]$ we observe that $\lk_G e$ spans at most three vertices $\{q,u,v\}$, again a contradiction.
We can therefore assume that $G[\overline{P_1},N_q\cap C_2]$ is complete bipartite and $G[P_2,(N_q\cap C_2)\setminus \{v_1,v_2\}]$ has no edges.

For every edge $f\in G[\overline{P_2}]$ the link $\lk_Gf$ misses $q$ and $N_q\setminus\{v_1,v_2\}$ hence it must contain $q'$. We therefore have that 
\begin{equation}\label{eq:un}
\overline{P_2}\subset N_{q'}\;. 
\end{equation}

The rest of the proof depends on whether $N_{q'}\cap P_1$ is empty. 

First suppose that $q'$ is adjacent to some vertex of $P_1$. Recalling that $N_{q'}\cap C_1\neq C_1$ and combining this with~\eqref{eq:un} we have $N_{q'}\cap P_1\neq P_1$. We can find $t\in P_1$ with neighbors $t',t''\in \overline{P_1}$ such that $tq'\in E(G)$ and $t'q'\not\in E(G)$. Since $\lk_Gtt'$ contains neither $q$ nor $q'$ it must be all of $C_2$ hence $C_2\subset N_t$. We then have $\lk_Gq't=G[\{t''\}\cup(N_{q'}\cap C_2)]$, so $N_{q'}\cap C_2$ induces a path within $C_2$ and $t''$ is not adjacent to its internal vertices. Since $v_1,v_2\in N_{q'}\cap C_2$ we obtain that $N_{q'}\cap C_2=(C_2\setminus N_{q})\cup\{v_1,v_2\}$. 

Let $|C_1|=k, |C_2|=l$. Subtracting the edges we lose from $P_2$ to $(N_q\cap C_2)\setminus\{v_1,v_2\}$ and from $t''\in \overline{P_1}$ to $C_2\setminus N_q$ and using $\deg(q',C_1)\leq k-1$, $|P_2|\geq 1$ and $a\geq 3$ we get
\begin{align*}
e(G)&\leq kl+k+l+(a+2)+(l-a+2+k-1)+1-|P_2|(a-2)-(l-a)\\
&\leq kl+2k+l+6.
\end{align*}

Next consider the case $N_{q'}\cap P_1=\emptyset$. By the usual argument we have $C_2\subset N_u,N_v$. Let $s\in P_2$ be the neighbor of $v$. Then $\lk_Gq'v=G[\{s,q\}\cup(N_{q'}\cap C_2)]$ and it contains the edges $v_1qv_2$. It follows that there are vertices $w_1,w_2\in C_2$ such that $G[N_{q'}\cap C_2]$ has two parts, stretching from $v_1$ to $w_1$ and from $v_2$ to $w_2$ (possibly $w_1=v_1$ or $w_2=v_2$). Moreover, looking at $\lk_Gq'v$ we see that $sw_1,sw_2\in E(G)$ but $s$ is not adjacent to the vertices in $(N_{q'}\cap C_2)\setminus\{w_1,w_2\}$.

Let $b=|N_{q'}\cap C_2|$. Counting the missing edges from $P_2$ to $(N_q\cap C_2)\setminus\{v_1,v_2\}$ and the disjoint set of missing edges from $s$ to $(N_{q'}\cap C_2)\setminus\{w_1,w_2\}$ we have:
\begin{align*}
e(G)&\leq kl+k+l+(a+2)+(b+k-1)+1-|P_2|(a-2)-(b-2)\\
&\leq kl+2k+l+6.
\end{align*}

An application of Lemma~\ref{stupid-inequality} completes the proof.
\end{proof}

Putting the above results (Propositions~\ref{prop:2-join-deg-01}, \ref{prop:2-join-deg-2-adjacent} and \ref{prop:2-join-deg-2-non-adjacent}) together we get the main result of this section concerning 2-joinlike graphs.
\begin{proposition}
\label{prop:2-join-all}
If $G$ is $2$-joinlike with exceptional vertices $\{q,q'\}$ and $\deg(q,C_1)\leq 2$ then $e(G)\leq \frac{1}{4}(n^2+2n+17)$ where $n=|V(G)|$.
\end{proposition}

\section{Closing remarks}
\begin{center}
\begin{figure}
\begin{tabular}{cc}
\includegraphics[scale=0.3]{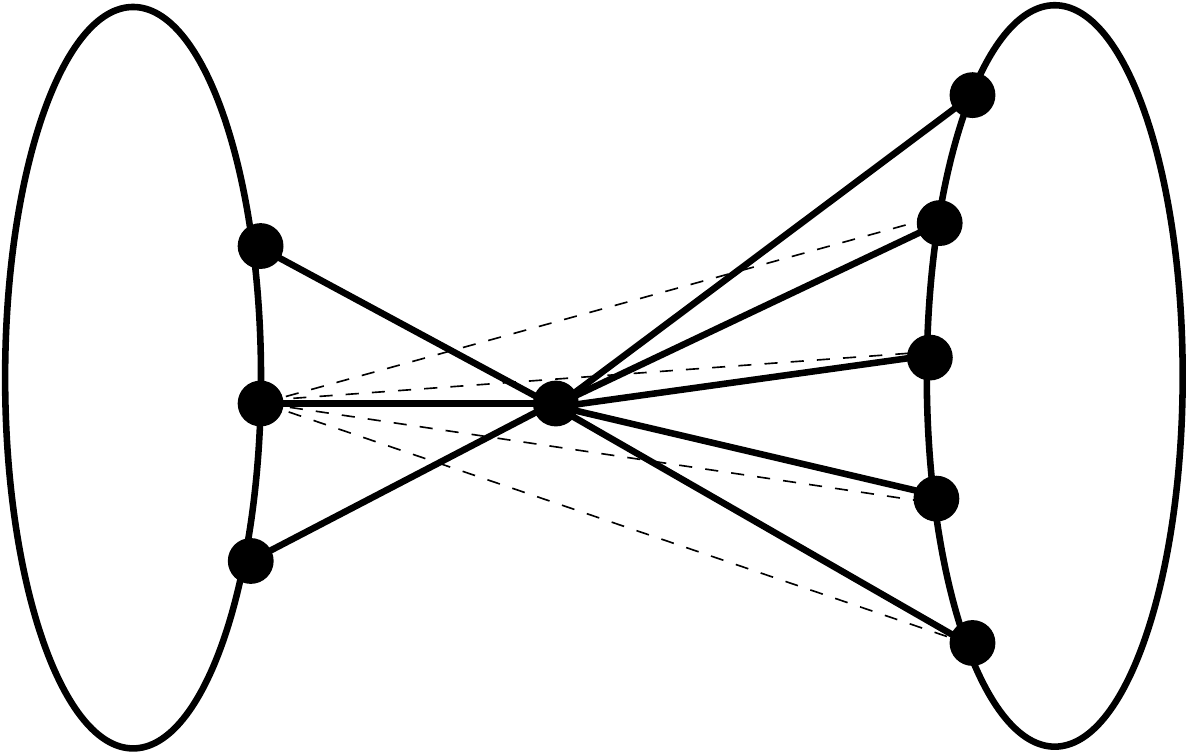} & \includegraphics[scale=0.3]{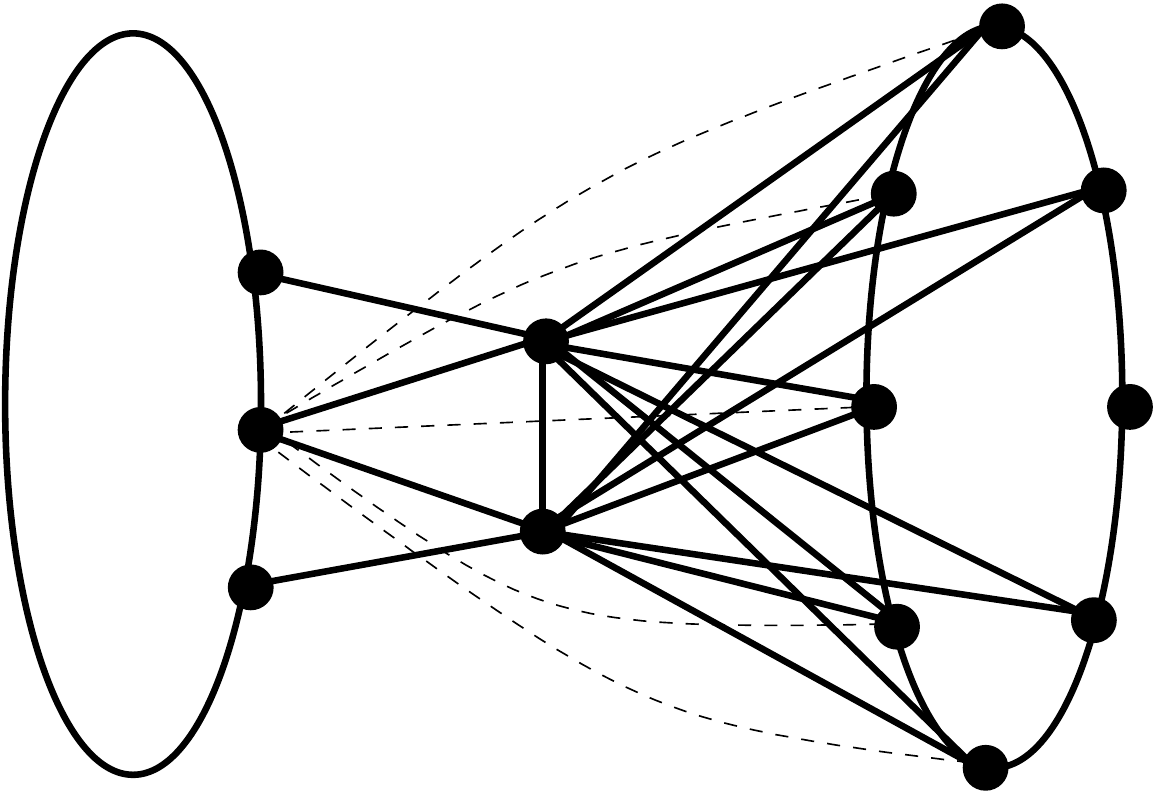} \\
a) & b)
\end{tabular}
\caption{The $1$-skeleta of two triangulations of $S^3$ with $f_1=\frac{1}{4}(f_0^2+2f_0+17)$. Starting from the join of two cycles remove the dashed edges and add the exceptional point(s) with the solid edges. In a) $|C_1|=|C_2|$ and $\deg(q,C_1)=3$. In b) $|C_2|=|C_1|+1$, $\deg(q,C_1)=\deg(q',C_1)=2$ and $\deg(q,C_2)=\deg(q',C_2)=|C_2|-1$.\label{rysunek-equal}}
\end{figure}
\end{center}

A careful analysis of the proofs in Section~\ref{section:exact} reveals two families of fascinating graphs which satisfy the equality $m=\frac{1}{4}(n^2+2n+17)$ for $n\geq n_0$. They appear in Proposition~\ref{prop:1-join-deg-3} and Proposition~\ref{prop:2-join-deg-2-adjacent}, see Figure~\ref{rysunek-equal}. This proves the claim made in Remark~\ref{remark:funnygraphs}; we omit the details.

Let us finish by stating a generalization of Theorem~\ref{thm:main} to higher dimensions.

\begin{conjecture}\label{conj:generalDIM}
For every $s\geq 2$ there exists a number $n_0=n_0(s)$ such that the following holds. 
If $M$ is a closed flag $(2s-1)$-manifold or a flag $(2s-1)$-GHS with $f_0\geq n_0$ vertices and $f_1$ edges then
\begin{equation}\label{eq:maxi}
f_1\leq f_0^2\cdot\frac{s-1}{2s}+f_0.
\end{equation}

Moreover, if $M$ satisfies
\begin{equation}\label{eq:maxi-2}
f_1> f_0^2\cdot\frac{s-1}{2s}+f_0\cdot\frac{s-1}{s}+\frac{7s+3}{2s}
\end{equation}
then $M$ is a join of $s$ polygons, in particular it is homeomorphic to $S^{2s-1}$.
\end{conjecture}

The maximal value in \eqref{eq:maxi} is achieved by the balanced join of $s$ cycles of lengths $f_0/s$. The expression in \eqref{eq:maxi-2} is the number of edges in the single edge-subdivision of such a join.

Let us sketch how one might prove this conjecture (the details will appear elsewhere). Fix $s\geq 2$ and denote $n=f_0$. First of all, $M$ is Eulerian and the ``middle'' Dehn-Sommerville equation $h_{s-1}=h_{s+1}$ can be rewritten in the form
$$f_s=sf_{s-1}+a_2f_{s-2}+\cdots+a_{s}f_0$$
for some coefficients $a_i$ depending only on $s$. It follows that the number of $(s+1)$-cliques in the $1$-skeleton $G=M^{(1)}$ is only $O(n^s)$. However, the number of edges in $G$ is above the Tur\'an bound for a complete, balanced $s$-partite graph, which is the maximizer of the number of edges among $K_{s+1}$-free graphs. By an application of the stability method we get that $G$ looks very similar to $K_{\ell,\ell,\ldots,\ell}$, where $\ell=n/s$. Next, as in the case of fascinating graphs, we see that in $G$ the link of every $(2s-1-j)$-clique is a triangulation of $S^j$ for $j=0,1,2$ (or for all $0\leq j\leq 2s-2$ if $M$ is a manifold) and one can try to exploit those conditions to rigidify the structure of $G$.

\subsection*{Acknowledgement.} We thank Anna Adamaszek, Martina Kubitzke and Eran Nevo for discussions on this and related topics and the anonymous referee for a number of remarks, including a correction to the proof of Claim~\ref{claim:GqqCdetermined}.

\bibliographystyle{alpha}
\bibliography{bibl}
\end{document}